\documentclass[10pt,reqno]{amsart}

\usepackage{a4wide}
\usepackage{amsmath,amssymb}
\usepackage[all]{xy}
\usepackage{mathbbol}
\usepackage{mathrsfs}
\usepackage{hyperref}
\usepackage{mathdots}

\newtheorem{proposition}{Proposition}[section]
\newtheorem{theorem}[proposition]{Theorem}
\newtheorem{lemma}[proposition]{Lemma}
\newtheorem{corollary}[proposition]{Corollary}

\newcommand{\cst}{\ifmmode\mathrm{C}^*\else{$\mathrm{C}^*$}\fi}
\newcommand{\st}{\;\vline\;}
\newcommand{\tens}{\otimes}
\newcommand{\id}{\mathrm{id}}
\newcommand{\comp}{\!\circ\!}
\newcommand{\I}{\mathbb{1}}
\newcommand{\ph}{\varphi}
\newcommand{\CC}{\mathbb{C}}
\newcommand{\NN}{\mathbb{N}}
\newcommand{\ZZ}{\mathbb{Z}}
\newcommand{\balpha}{\boldsymbol{\alpha}}
\newcommand{\btau}{\boldsymbol{\tau}}
\newcommand{\bu}{\boldsymbol{u}}
\newcommand{\GG}{\mathbb{G}}

\DeclareMathOperator{\Mor}{Mor}
\DeclareMathOperator{\lin}{\mathrm{span}}
\DeclareMathOperator{\bb}{\mathfrak{b\!}}

\renewcommand{\Hat}[1]{\widehat{#1}}
\newcommand{\KK}{\mathbb{K}}
\newcommand{\cH}{\mathscr{H}}
\newcommand{\cA}{\mathscr{A}}
\newcommand{\cB}{\mathscr{B}}
\newcommand{\cC}{\mathscr{C}}
\DeclareMathOperator{\cM}{\mathscr{M}}
\DeclareMathOperator{\QISO}{\widetilde{\mathsf{QISO}}{}^+\!\!\;}
\DeclareMathOperator{\comb}{comb}

\numberwithin{equation}{section}

\begin{document}

\title{Quantum isometry groups of symmetric groups}

\date{\today}

\author{Jan Liszka-Dalecki}
\address{Department of Mathematical Methods in Physics, Faculty of Physics, University of Warsaw
}

\author{Piotr M.~So{\l}tan}
\address{Department of Mathematical Methods in Physics, Faculty of Physics, University of Warsaw\newline
\indent{and}\newline
\indent{}Institute of Mathematics, Polish Academy of Sciences}
\email{piotr.soltan@fuw.edu.pl}
\urladdr{http://www.fuw.edu.pl/~psoltan/en/}

\thanks{Partially supported by Polish government grant no.~N201 1770 33, European Union grant PIRSES-GA-2008-230836 and Polish government matching grant no.~1261/7.PR UE/2009/7.}

\begin{abstract}
We identify the quantum isometry groups of spectral triples built on the symmetric groups with length functions arising from the nearest-neighbor transpositions as generators. It turns out that they are isomorphic to certain ``doubling'' of the group algebras of the respective symmetric groups. We discuss the doubling procedure in the context of regular multiplier Hopf algebras. In the last section we study the dependence of the isometry group of $S_n$ on the choice of generators in the case $n=3$. We show that two different choices of generators lead to non-isomorphic quantum isometry groups which exhaust the list of non-commutative non-cocommutative semisimple Hopf algebras of dimension 12. This provides non-commutative geometric interpretation of these Hopf algebras.
\end{abstract}


\maketitle

\section{Introduction and main results}\label{intro}

Quantum isometry groups of non-commutative manifolds have been introduced in \cite{gos}. Given a spectral triple $(\cA,\cH,D)$ one considers the category of all quantum families (in the sense of \cite{pseu,qs}) of orientation preserving isometries of $(\cA,\cH,D)$. When this category has a universal object, it can be shown that this object is in fact a compact quantum group. This object is denoted by $\QISO(\cA,\cH,D)$ and is called the quantum group of orientation preserving isometries of $(\cA,\cH,D)$. We refer to \cite{gos,bg1,bg2} and also \cite{bs} for details of the theory of quantum isometry groups and numerous examples.

This paper is devoted to the study of an example in the theory of compact quantum groups. For each natural $n$ we will find a universal compact quantum group coacting on the group algebra of the symmetric group $S_n$ satisfying certain additional property related to the theory of non-commutative manifolds which corresponds to the concept of a smooth isometric group action preserving orientation of an oriented Riemannian manifold (\cite{gos,bg2,bg1}). In recent papers \cite{bgPod,so3} interesting interpretation of the universal quantum group acting isometrically on some non-commutative manifold was discovered.

For the theory of compact quantum groups and their relation to Hopf algebras we refer e.g.~to the exposition \cite{cqg}. In Section \ref{doubling} we will use the language of multiplier Hopf algebras developed in \cite{mha,afgd}.

The primary object of this paper is to give a full description of the quantum isometry groups $\QISO\bigl(\CC[S_n],\ell^2(S_n),D\bigr)$, where $S_n$ is the symmetric group and $D$ is the Dirac operator given by multiplication by the \emph{length function} on $S_n$ associated to the set
\begin{equation}\label{S}
\bigl\{s_1,\dotsc,s_{n-1}\bigr\}
\end{equation}
of generators of $S_n$; here and throughout the paper the symbol $s_i$ will denote the transposition $(i,i+1)$ with $i\in\{1,\dotsc,n-1\}$. The length function, by definition, assigns to each element of $x\in{S_n}$ the minimal number of transpositions from \eqref{S} needed to write $x$ as a product of these elements.

Our work is a direct extension of the effort to compute the quantum isometry groups of spectral triples associated to various group \cst-algebras initiated in \cite{bs} and continued e.g.~in \cite{bansk,bans}. The simplest cases of $S_2$ and $S_3$ were already treated in \cite[Sections 3 \& 4]{bs}. Our main result concerns only the fixed set of generators \eqref{S}. See however Section \ref{rank} for a result related to different sets of generators.

Before continuing let us declare that all algebras considered in this paper are $*$-algebras over $\CC$. By an automorphism we shall always mean a $*$-automorphism. We will use some very basic properties of \cst-algebras. For the theory of operator algebras (and in particular \cst-algebras) we refer to classic texts like \cite{dix,tak,ped}.

\subsection{Isometries of the spectral triple on the symmetric group}\label{SpTr}

The spectral triple
\begin{equation*}
\bigl(\CC[S_n],\ell^2(S_n),D\bigr)
\end{equation*}
is very well behaved (mainly because $\CC[S_n]$ is finite dimensional). The notion of a quantum group acting isometrically on $\bigl(\CC[S_n],\ell^2(S_n),D\bigr)$ can in this case be rewritten as in \cite[Section 2]{bs}: let $\GG=(A,\Delta_A)$ be a compact quantum group and let $\alpha\colon\CC[S_n]\to\CC[S_n]\tens{A}$ be a coaction of $\GG$ on $\CC[S_n]$ satisfying the \emph{Podle\'s condition}, i.e.~the set
\begin{equation*}
\alpha\bigl(\CC[S_n]\bigr)(\I\tens{A})=\lin\bigl\{\alpha(x)(\I\tens{a})\st{x}\in\CC[S_n],\:a\in{A}\bigr\}
\end{equation*}
is dense in $\CC[S_n]\tens{A}$ (cf.~\cite{podles,apqs}).

The coaction $\alpha$ is \emph{isometric} if
\begin{enumerate}
\item\label{Taup} for all $x\in\CC[S_n]$ we have $(\btau\tens\id)\alpha(x)=\btau(x)\I$, where $\btau$ is the canonical trace on $\CC[S_n]$ (i.e.~the evaluation at the neutral element $e\in{S_n}$ of an arbitrary $x\in\CC[S_n]$ considered as a function on $S_n$),
\item\label{Dal} we have $\alpha\comp\Hat{D}=(\Hat{D}\tens\I)\comp\alpha$, where $\Hat{D}$ is the operator $D$ considered as a map $\CC[S_n]\to\CC[S_n]$ (instead of $\ell^2(S_n)\to\ell^2(S_n)$).
\end{enumerate}
It is, in fact, easily seen that \eqref{Dal} implies \eqref{Taup}.

In the formulation of the main theorem of this paper we will use the notion of doubling of a multiplier Hopf algebra with invariant integrals. This notion is discussed in detail in Section \ref{doubling}. Let us mention here that this construction is analogous to the operation of taking semidirect product by $\ZZ_2$ in group theory.

Theorem 2.6 of \cite{bs} states that there exists a universal object in the category of all compact quantum groups acting isometrically on $\bigl(\CC[S_n],\ell^2(S_n),D\bigr)$ and it is precisely $\QISO\bigl(\CC[S_n],\ell^2(S_n),D\bigr)$. We will use this fact to prove the following:

\begin{theorem}\label{main}
The quantum isometry group $\QISO\bigl(\CC[S_n],\ell^2(S_n),D\bigr)$ of the spectral triple on $S_n$ is isomorphic to the finite quantum group obtained as doubling of the group algebra $\CC[S_n]$ with standard cocommutative Hopf algebra structure. In particular the \cst-algebra of continuous functions on the quantum isometry group $\QISO\bigl(\CC[S_n],\ell^2(S_n),D\bigr)$ is isomorphic to $\CC[S_n]\oplus\CC[S_n]$. If $\{\sigma_1,\dotsc,\sigma_{n-1}\}$ and $\{\tau_1,\dotsc,\tau_{n-1}\}$ denote the generators \eqref{S} in the first and second copy of $\CC[S_n]$ inside $\CC[S_n]\oplus\CC[S_n]$ respectively then the comultiplication $\Delta$ of $\QISO\bigl(\CC[S_n],\ell^2(S_n),D\bigr)$ is given by
\begin{equation*}
\begin{split}
\Delta(\sigma_i)&=\sigma_i\tens\sigma_i+\tau_i\tens\tau_{n-i},\\
\Delta(\tau_i)&=\sigma_i\tens\tau_i+\tau_i\tens\sigma_{n-i}.
\end{split}
\end{equation*}
\end{theorem}

Let us note that the same phenomenon of the quantum isometry group being isomorphic to the doubling of the original Hopf algebra was already observed in several cases (e.g.~for singly generated groups in \cite[Section 3]{bs}; the same result for the infinite dihedral group $D_\infty$ was communicated to us by A.~Skalski), but is not of universal nature (\cite[Section 5]{bs}).

The final remark of this section concerns possibly non-unital \cst-algebras. The notion of a morphism of such \cst-algebras was introduced already in \cite{pseu} and developed further in e.g.~\cite{unbo}. We refer the reader to the latter publication for the precise definition of the set $\Mor(C,B)$ of morphisms between two arbitrary \cst-algebras $C$ and $B$.

Note also that the conditions \eqref{Taup} and \eqref{Dal} defining an isometric coaction of a compact quantum group on $\bigl(\CC[S_n],\ell^2(S_n),D\bigr)$ make sense for an arbitrary quantum family of maps. More precisely, if $C$ is any \cst-algebra and $\Psi\in\Mor\bigl(\CC[S_n],\CC[S_n]\tens{C}\bigr)$ then we may say that $\Psi$ defines a quantum family of isometries of $\bigl(\CC[S_n],\ell^2(S_n),D\bigr)$ if
\begin{itemize}
\item $(\btau\tens\id)\Psi(x)=\btau(x)\I$, for all $x\in\CC[S_n]$,
\item $\Psi\comp\Hat{D}=(\Hat{D}\tens\I)\comp\Psi$.\footnote{These conditions make sense despite the fact that the image of the morphism $\Psi$ might lie outside $\CC[S_n]\tens{C}$, cf.~\cite{apqs}.}
\end{itemize}
In particular one can take for $C$ a \cst-algebra corresponding to a quantum semigroup $\mathbb{S}=(C,\Delta_C)$. Then, if $\Psi$ is a coaction of $\mathbb{S}$ satisfying the Podle\'s condition, by results of \cite{qbohr} and \cite[Theorem 2.1]{apqs}, we know that there exists a compact quantum group $\bb{\mathbb{S}}=(B,\Delta_B)$ equipped with a quantum semigroup morphism $\chi\in\Mor(B,C)$ and a coaction $\bb\Psi$ of $\mathbb{S}$ on $\CC[S_n]$ such that
\begin{equation*}
\Psi=(\id\tens\chi)\comp\bb\Psi.
\end{equation*}
Moreover $\bb\Psi$ also preserves $\btau$. It is easy to show that this new coaction is an isometric coaction of $\bb\mathbb{S}$ on $\bigl(\CC[S_n],\ell^2(S_n),D\bigr)$. All this shows that considering only compact quantum group coactions is not really restrictive: all quantum semigroup coactions (even non-compact ones) satisfying the Podle\'s condition can always be realized by \emph{compact quantum group} coactions.

In the next section we obtain the full description of the isometry groups of the spectral triples on symmetric groups discussed above. Then, in Section \ref{doubling} we define the doubling procedure and show that, in the special case of a group algebra, the resulting object is indeed a certain doubling of the original. Moreover, when the procedure is applied to the standard cocommutative Hopf algebra structure on $\CC[S_n]$, the doubling turns out to be precisely the quantum group $\QISO\bigl(\CC[S_n],\ell^2(S_n),D\bigr)$. In the last section we deal with the question left open in \cite{bs} whether the quantum isometry group depends on the choice of the set of generators. We show that for two different generating sets for $S_3$ we obtain two different Hopf algebras. By results of \cite{F} these exhaust the list of non-commutative non-cocommutative semisimple Hopf algebras of dimension 12.

\section{Quantum isometry group of \texorpdfstring{$\bigl(\CC[S_n],\ell^2(S_n),D\bigr)$}{(C[G],l2(G),D)}}
\label{qiso}

Let $\GG=(A,\Delta_A)$ be a compact quantum group and let
\begin{equation*}
\alpha\colon\CC[S_n]\longrightarrow\CC[S_n]\tens{A}
\end{equation*}
be a coaction of $\GG$ on $\CC[S_n]$ satisfying the Podle\'s condition. Then, due to finite dimensionality of $\CC[S_n]$, the image of $\alpha$ is contained in the (algebraic) tensor product $\CC[S_n]\tens\cA$, where $\cA$ is the Hopf $*$-algebra canonically associated with $\GG$ (cf.~\cite{podles,podMu,cqg}. Moreover $\alpha$ is a coaction of the Hopf $*$-algebra $(\cA,\Delta_{\cA})$. In particular we have
\begin{equation}\label{counit}
(\id\tens\epsilon)\comp\alpha=\id,
\end{equation}
where $\epsilon$ is the counit of $(\cA,\Delta_{\cA})$.

Now we shall make two assumptions on the coaction $\alpha$:
\begin{itemize}
\item the coaction preserves the subspace spanned by elements of length $1$, i.e.~for any $j\in\{1,\dotsc,n-1\}$ there exist elements $u_{1,j},\dotsc,u_{n-1,j}\in\cA$ such that
\begin{equation}\label{alsj}
\alpha(s_j)=\sum_{i=1}^{n-1}s_i\tens{u_{i,j}}.
\end{equation}
\item the canonical trace $\btau$ on $\CC[S_n]$ is preserved by $\alpha$ on elements of length $0$ and $2$:
\begin{equation}\label{deltakl}
(\btau\tens\id)\alpha(s_ks_l)=
\begin{cases}
\I&k=l,\\0&\text{otherwise}.
\end{cases}
\end{equation}
\end{itemize}
From the above conditions we shall derive a number of properties of $(A,\Delta_A)$ which will help us determine the quantum isometry group of the spectral triple described in Section \ref{intro}. Note that if $\alpha$ is an isometric coaction (as defined in Subsection \ref{SpTr}) then it certainly satisfies these two conditions.

Let us denote by $\bu$ the matrix
\begin{equation*}
\begin{bmatrix}
u_{1,1}&\dotsm&u_{1,n-1}\\
\vdots&\ddots&\vdots\\
u_{n-1,1}&\dotsm&u_{n-1,n-1}
\end{bmatrix}.
\end{equation*}
All matrix elements of $\bu$ are selfadjoint. This is because the generators $\{s_1,\dotsc,s_{n-1}\}$ are selfadjoint, so that for each $j$
\begin{equation*}
\sum_{i=1}^{n-1}s_i\tens{u_{i,j}^*}=\alpha(s_j)^*=\alpha(s_j)=\sum_{i=1}^{n-1}s_i\tens{u_{i,j}}.
\end{equation*}
Now the equality $u_{i,j}^*=u_{i,j}$ follows from linear independence of $\{s_1,\dotsc,s_{n-1}\}$.

\begin{proposition}\label{unitary}
The matrix $\bu$ is unitary.
\end{proposition}

\begin{proof}
We have
\begin{equation*}
\begin{split}
(\btau\tens\id)\alpha(s_ks_l)&=(\btau\tens\id)\biggl(\biggl(\,\sum_{i=1}^{n-1}{s_i}\tens{u_{i,k}}\biggr)
\biggl(\,\sum_{j=1}^{n-1}{s_j}\tens{u_{j,l}}\biggr)\biggr)\\
&=(\btau\tens\id)\biggl(\,\sum_{i,j=1}^{n-1}{s_is_j}\tens{u_{i,k}u_{j,l}}\biggr)\\
&=\sum_{i,j=1}^{n-1}\btau(s_is_j)u_{i,k}u_{j,l}
=\sum_{i=1}^{n-1}u_{i,k}u_{i,l}.
\end{split}
\end{equation*}
We also know that all elements $u_{i,j}$ are selfadjoint, so that by formula \eqref{deltakl} we have
\begin{equation}\label{bustbu}
\bu^*\bu=\I.
\end{equation}
On the other hand we know that the matrix $\bu$ is invertible and the inverse is given by
\begin{equation*}
(\id\tens{S})\bu=
\begin{bmatrix}
S(u_{1,1})&\dotsm&S(u_{1,n-1})\\
\vdots&\ddots&\vdots\\
S(u_{n-1,1})&\dotsm&S(u_{n-1,n-1})
\end{bmatrix}.
\end{equation*}
where $S$ is the antipode of $(\cA,\Delta_{\cA})$. Indeed, since $\alpha$ is a coaction, we have
\begin{equation*}
\begin{split}
\sum_{i=1}^{n-1}s_i\tens\Delta_{\cA}(u_{i,j})&=(\id\tens\Delta_{\cA})\alpha(s_j)\\
&=(\alpha\tens\id)\alpha(s_j)=\sum_{i=1}^{n-1}\alpha(s_i)\tens{u_{i,j}}=
\sum_{i=1}^{n-1}\sum_{k=1}^{n-1}s_k\tens{u_{k,i}}\tens{u_{i,j}}.
\end{split}
\end{equation*}
Therefore
\begin{equation*}
\Delta_{\cA}(u_{k,j})=\sum_{i=1}^{n-1}u_{k,i}\tens{u_{i,j}}.
\end{equation*}
The defining properties of the antipode of a Hopf algebra now show that
\begin{equation*}
\sum_{i=1}^{n-1}S(u_{k,i}){u_{i,j}}=
\sum_{i=1}^{n-1}u_{k,i}S(u_{i,j})=\epsilon(u_{k,j})\I.
\end{equation*}
The fact that
\begin{equation*}
\epsilon(u_{k,j})=
\begin{cases}
1&k=j,\\
0&\text{otherwise}
\end{cases}
\end{equation*}
follows from \eqref{alsj} and \eqref{counit}. Therefore $\bu$ is invertible with inverse $(\id\tens{S})\bu$ and by \eqref{bustbu} we have
\begin{equation*}
\bu^{-1}=\bu^*.
\end{equation*}
\end{proof}

Let us note a very useful corollary of the proof of Proposition \ref{unitary}:

\begin{corollary}
For all $i,j\in\{1,\dotsc,n-1\}$ we have
\begin{equation*}
\begin{split}
\Delta_{\cA}(u_{i,j})&=\sum_{k=1}^{n-1}u_{i,k}\tens{u_{k,j}},\\
S(u_{i,j})&=u_{j,i}.
\end{split}
\end{equation*}
\end{corollary}

Before continuing our analysis of the coaction $\alpha$ let us recall the following simple fact:

\begin{lemma} \label{zero_repr}
Let $C$ be a \cst-algebra and take $a,b\in{C}$ with $a=a^*$. Then $a^mb=0$ for some $m\in\NN$ implies $ab=0$.
\end{lemma}

\subsection{Implications of the relation \texorpdfstring{$s_i^2=e$}{si2=e}}

Fix $i\in\{1,\dotsc,n-1\}$. We have
\begin{equation*}
\alpha(s_i)^2=\alpha(s_i^2)=e\tens\I,
\end{equation*}
which means
\begin{equation*}
\sum_{j,k=1}^{n-1}s_js_k\tens{u_{j,i}u_{k,i}}=e\tens\I
\end{equation*}
and can be rewritten as
\begin{equation*}
\sum_{j=1}^{n-1}s_j^2\tens{u_{j,i}u_{j,i}}+\sum_{j<k}
\left(s_js_k\tens{u_{j,i}u_{k,i}}+s_ks_j\tens{u_{j,i}u_{k,i}}\right)=e\tens\I
\end{equation*}
Note that the only relations in $S_n$ holding between products of two generators are
\begin{equation}\label{dwael}
\begin{aligned}
s_i^2&=e&&\text{for all }i\in\{1,\dotsc,n-1\},\\
s_is_j&=s_js_i&&\text{for all }i,j\in\{1,\dotsc,n-1\}\text{ with }|i-j|>1.
\end{aligned}
\end{equation}
Remembering that the elements of a group taken as elements of the group algebra are linearly independent we get the following relations:
\begin{subequations}
\begin{align}
&\sum_{j}u_{j,i}^2=\I,\label{sumkw}\\
&u_{j,i}u_{j+1,i}=0=u_{j+1,i}u_{j,i},&&1\leq{j}\leq{n-2},\label{blpion}\\
&u_{j,i}u_{k,i}+u_{k,i}u_{j,i}=0,&&|j-k|>1.\label{dalpion}
\end{align}
\end{subequations}
Relation \eqref{sumkw} follows, in fact, from the unitarity of $\bu$ proved in Proposition \ref{unitary}. Applying the antipode to \eqref{blpion} and \eqref{dalpion} we get
\begin{subequations}
\begin{align}
&u_{i,j+1}u_{i,j}=0=u_{i,j}u_{i,j+1},&&1\leq{j}\leq{n-2},\label{blpoz}\\
&u_{i,k}u_{i,j}+u_{i,j}u_{i,k}=0,&&|j-k|>1.\label{dalpoz}
\end{align}
\end{subequations}

\subsection{Implications of the relation \texorpdfstring{$s_is_j=s_js_i$}{sisj=sjsi} for \texorpdfstring{$|i-j|>1$}{|i-j|>1}}

Now we take $i,j\in\{1,\dotsc,n-1\}$ such that $|i-j|>1$. Then $\alpha(s_i)\alpha(s_j)=\alpha(s_j)\alpha(s_i)$. Therefore
\begin{equation*}
\biggl(\,\sum_{k=1}^{n-1}s_k\tens{u_{k,i}}\biggr)\biggl(\,\sum_{l=1}^{n-1}s_l\tens{u_{l,j}}\biggr)
=\biggl(\,\sum_{m=1}^{n-1}s_m\tens{u_{m,j}}\biggr)\biggl(\,\sum_{p=1}^{n-1}s_p\tens{u_{p,i}}\biggr),
\end{equation*}
i.e.
\begin{equation*}
\sum_{k,l=1}^{n-1}s_ks_l\tens{u_{k,i}}{u_{l,j}}=\sum_{m,p=1}^{n-1}s_ms_p\tens{u_{m,j}}{u_{p,i}}.
\end{equation*}

Recalling \eqref{dwael} and the linear independence of the generators in the group algebra we get
\begin{subequations}
\begin{align}
\sum_{k=1}^{n-1} u_{k,i}u_{k,j}&=\sum_{k=1}^{n-1} u_{k,j}u_{k,i},\label{suma_unit_dwa}\\
u_{k,i}u_{k+1,j}&=u_{k,j}u_{k+1,i},\label{konik1_podw}\\
u_{k+1,i}u_{k,j}&=u_{k+1,j}u_{k,i},\label{konik2_podw}\\
u_{k,i}u_{l,j}+u_{l,i}u_{k,j}&=u_{k,j}u_{l,i}+u_{l,j}u_{k,i},&&|k-l|>1.\label{krzyzyk}
\end{align}
\end{subequations}

Both sides of \eqref{suma_unit_dwa} are equal to 0 by unitarity of $\bu$. Moreover multiplying \eqref{konik1_podw} from the left by $u_{k,i}$ and using \eqref{dalpoz} and \eqref{blpion} we obtain
\begin{equation*}
u_{k,i}^2u_{k+1,j}=u_{k,i}u_{k,j}u_{k+1,i}=-u_{k,j}u_{k,i}u_{k+1,i}=0.
\end{equation*}
By Lemma \ref{zero_repr}, $u_{k,i}u_{k+1,j}=0$ and both sides of \eqref{konik1_podw} are actually $0$. We can do the same trick starting with \eqref{konik2_podw} thus
\begin{subequations}\label{Konik}
\begin{align}
u_{k,i}u_{k+1,j}&=0=u_{k,j}u_{k+1,i},\label{konik1}\\
u_{k+1,i}u_{k,j}&=0=u_{k+1,j}u_{k,i}\label{konik2}
\end{align}
\end{subequations}
for all $k\in\{1,\dotsc,n-2\}$.

Applying the antipode to \eqref{konik1}, \eqref{konik2} and \eqref{krzyzyk} yields
\begin{subequations}\label{Konik'}
\begin{align}
u_{j, k+1}u_{i,k}&=0=u_{i,k+1}u_{j,k},\label{konik1'}\\
u_{j,k}u_{i,k+1}&=0=u_{i,k}u_{j,k+1},\label{konik2'}\\
u_{j,l}u_{i,k}+u_{j,k}u_{i,l}&=u_{i,l}u_{j,k}+u_{i,k}u_{j,l},&&|k-l|>1.\label{krzyzyk'}
\end{align}
\end{subequations}
Putting $k=i$ and $l=j$ in equations \eqref{krzyzyk} and \eqref{krzyzyk'} and subtracting them we get
\begin{equation}\label{commut}
\begin{split}
u_{i,i}u_{j,j}&=u_{j,j}u_{i,i},\\
u_{i,j}u_{j,i}&=u_{j,i}u_{i,j}
\end{split}
\end{equation}
whenever $|i-j|>1$.

\subsection{Implications of the relation \texorpdfstring{$s_is_{i+1}s_i=s_{i+1}s_is_{i+1}$}{sisi+1si=si+1sisi+1}}

We now fix $i\in\{1,\dotsc,n-2\}$. We have $\alpha(s_i)\alpha(s_{i+1})\alpha(s_i)=\alpha(s_{i+1})\alpha(s_i)\alpha(s_{i+1})$, which reads
\begin{equation*}
\begin{split}
\biggl(\,
\sum_{j=1}^{n-1}s_j\tens{u_{j,i}}
\biggr)
&\biggl(\,
\sum_{k=1}^{n-1}s_k\tens{u_{k,i+1}}
\biggr)
\biggl(\,
\sum_{l=1}^{n-1}s_l\tens{u_{l,i}}
\biggr)\\
&=
\biggl(\,
\sum_{p=1}^{n-1}s_p\tens{u_{p,i+1}}
\biggr)
\biggl(\,
\sum_{q=1}^{n-1}s_q\tens{u_{q,i}}
\biggr)
\biggl(\,
\sum_{r=1}^{n-1}s_r\tens{u_{r,i-1}}
\biggr)
\end{split}
\end{equation*}
in coordinates. Thus
\begin{equation*}
\sum_{j,k,l}s_js_ks_l\tens{u_{j,i}u_{k,i+1}u_{l,i}}
=\sum_{p,q,r}s_ps_qs_r\tens{u_{p,i+1}u_{q,i}u_{r,i+1}}.
\end{equation*}
It is easy to see there is only one three letter word representing the element $x=s_qs_{q+1}s_{q+2}$. It follows that
\begin{equation} \label{haczyk_podw}
u_{q,i}u_{q+1,i+1}u_{q+2,i}=u_{q,i+1}u_{q+1,i}u_{q+2,i+1}.
\end{equation}
Similarly, there is only one three letter word representing the element $x=s_{q+2}s_{q+1}s_{q}$. Thus
\begin{equation} \label{haczyk_podw'}
u_{q+2,i}u_{q+1,i+1}u_{q,i}=u_{q+2,i+1}u_{q+1,i}u_{q,i+1}.
\end{equation}

Multiplying both sides of \eqref{haczyk_podw} from the left by $u_{q,i}$ and using \eqref{blpoz} we obtain
\begin{equation*}
u_{q,i}^2u_{q+1,i+1}u_{q+2,i}=u_{q,i}u_{q,i+1}u_{q+1,i}u_{q+2,i+1}=0,
\end{equation*}
so again, by Lemma \ref{zero_repr} (and again by \eqref{haczyk_podw}), we have
\begin{equation} \label{haczyk}
u_{q,i}u_{q+1,i+1}u_{q+2,i}=0=u_{q,i+1}u_{q+1,i}u_{q+2,i+1}.
\end{equation}
Similarly, starting from \eqref{haczyk_podw'}, we find that
\begin{equation} \label{haczyk'}
u_{q+2,i}u_{q+1,i+1}u_{q,i}=0=u_{q+2,i+1}u_{q+1,i}u_{q,i+1}.
\end{equation}
Applying the antipode to \eqref{haczyk} and \eqref{haczyk'} yields
\begin{subequations}
\begin{align}
u_{i,q+2}u_{i+1,q+1}u_{i,q}&=0=u_{i+1,q+2}u_{i,q+1}u_{i+1,q},\label{haczyk_poz}\\
u_{i,q}u_{i+1,q+1}u_{i,q+2}&=0=u_{i+1,q}u_{i,q+1}u_{i+1,q+2}.\label{haczyk_poz'}
\end{align}
\end{subequations}
These are, of course, not all the relations on matrix elements of $\bu$ that follow from the braid relations on the generators \eqref{S}. We shall return to this problem in Subsection \ref{here}.

\subsection{Determination of \texorpdfstring{$\QISO\bigl(\CC[S_n],\ell^2(S_n),D\bigr)$}{QISO+(C[Sn],l2(Sn),D)}}\label{here}

\begin{proposition}
The only potentially non-zero elements of the matrix $\bu$ are on the diagonal and anti-diagonal.
\end{proposition}

\begin{proof}
Take $k\in\{2,\dotsc,n-2\}$. We will now show that $u_{1,k}=0$. First of all, by \eqref{sumkw}, we have
\begin{equation*}
u_{1,k}=\sum_{i=1}^{n-1}u_{i,k-1}^2u_{1,k}.
\end{equation*}
Now \eqref{konik2'} shows that the sum above has at most two non-zero terms:
\begin{equation*}
u_{1,k}=u_{1,k-1}^2u_{1,k}+u_{2,k-1}^2u_{1,k}.
\end{equation*}
Furthermore, by \eqref{blpoz}, the first of these terms is zero, so
\begin{equation*}
u_{1,k}=u_{2,k-1}^2u_{1,k}.
\end{equation*}
Multiplying both sides of this relation from the right by $u_{2,k+1}$ we get
\begin{equation*}
u_{1,k}u_{2,k+1}=u_{2,k-1}(u_{2,k-1}u_{1,k}u_{2,k+1}),
\end{equation*}
which is zero by \eqref{haczyk_poz'}. Thus
\begin{equation}\label{skos}
u_{1,k}u_{2,k+1}=0.
\end{equation}
Now, again by \eqref{sumkw}, we have
\begin{equation*}
u_{1,k}=\sum_{i=1}^{n-1}u_{1,k}u_{i,k+1}^2
\end{equation*}
and the sum contains only one the term (by \eqref{konik2'} and \eqref{blpoz}):
\begin{equation*}
u_{1,k}=u_{1,k}u_{2,k+1}^2=(u_{1,k}u_{2,k+1})u_{2,k+1}.
\end{equation*}
But by \eqref{skos} the product in parentheses is zero, so $u_{1,k}=0$.

Applying the antipode we also find that
\begin{equation*}
u_{k,1}=0,\qquad{2}\leq{k}\leq{n-2}.
\end{equation*}

Using the same technique to treat elements in the last row of $\bu$ we also conclude that
\begin{equation*}
u_{k,n-1}=u_{n-1,k}=0,\qquad{2}\leq{k}\leq{n-2}.
\end{equation*}
This means that
\begin{equation}\label{uup}
\bu=
\begin{bmatrix}
u_{1,1}&\begin{matrix}0\;\;&\dotsm&\;\;0\end{matrix}&u_{1,n-1}\\
\begin{matrix}0\\\vdots\\0\end{matrix}&
\left[\begin{smallmatrix}
u_{2,2}&\dotsm&u_{2,n-2}\\
\vdots&\ddots&\vdots\\
u_{n-2,2}&\dotsm&u_{n-2,n-2}
\end{smallmatrix}\right]
&\begin{matrix}0\\\vdots\\0\end{matrix}\\
u_{n-1,1}&\begin{matrix}0\;\;&\dotsm&\;\;0\end{matrix}&u_{n-1,n-1}
\end{bmatrix}
\end{equation}
Consider now the sub-matrix $\boldsymbol{u'}$ of $\bu$ chosen as indicated in \eqref{uup}. Then the technique used to deal with the first and last columns and rows of $\bu$ applies in exactly the same way to $\boldsymbol{u'}$ (note that the sum of squares of elements of rows/columns of $\boldsymbol{u'}$ is clearly equal to $\I$). Thus we can proceed by induction and, depending on the parity of $n$, arrive at either a $2\times{2}$ or $1\times{1}$ matrix. In both cases it is comprised solely of diagonal and anti-diagonal elements.
\end{proof}

We now know that the matrix $\bu$ must have the following special form depending on the parity of $n$:
\begin{center}
\bigskip
\begin{tabular}{c@{\quad}|@{\quad}c}
$n=2p+1$&$n=2p$\\
\hline
\\
$\bu=\left[\begin{smallmatrix}
a_1   &        &       &       &       &       &        &b_1   \\
      &a_2     &       &       &       &       &b_2     &      \\
      &        &\ddots &       &       &\iddots&        &      \\
      &        &       &a_p    &b_p    &       &        &      \\
      &        &       &b_{p+1}&a_{p+1}&       &        &      \\
      &        &\iddots&       &       &\ddots &        &      \\
      &b_{2p-1}&       &       &       &       &a_{2p-1}&      \\
b_{2p}&        &       &       &       &       &        &a_{2p}
\end{smallmatrix}\right]$
&
$\bu=\left[\begin{smallmatrix}
a_1           &       &       &     &       &       &b_1    \\
              &\ddots &       &     &       &\iddots&       \\
              &       &a_{p-1}&\quad&b_{p-1}&       &       \\
\phantom{\mid}&       &       &c    &       &       &       \\
              &       &b_{p+1}&     &a_{p+1}&       &       \\
              &\iddots&       &     &       &\ddots &       \\
b_{2p-1}      &       &       &     &       &       &a_{2p-1}
\end{smallmatrix}\right]$\\
\\
\end{tabular}
\end{center}
where $\{a_1,\dotsc,a_{2p},b_1,\dotsc,b_{2p}\}$ (and
$\{a_1,\dotsc,a_{p-1},a_{p+1},\dotsc,a_{2p-1},b_1,\dotsc,b_{p-1},b_{p+1},\dotsc,b_{2p-1},c\}$ respectively) are certain elements of $\cA$. We will use these symbols until the end of this section.

In the notation introduced above, the formula for the coaction $\alpha$ simplifies to:
\begin{equation}\label{alab}
\alpha(s_i)=s_i\tens{a_i}+s_{n-i}\tens{b_{n-i}}
\end{equation}
(valid for for all $i$ if $n=2p+1$, and for $i\neq{p}$ for $n=2p$). In the situation $n=2p$, $i=p$ we have
\begin{equation*}
\alpha(s_p)=s_p\tens{c}.
\end{equation*}

Let us investigate some consequences of this new formula for $\alpha$ in the case $n=2p$. We have $\alpha(s_ps_{p+1}s_p)=\alpha(s_{p+1}s_ps_{p+1})$, so
\begin{equation*}
\begin{split}
s_ps_{p+1}s_p\tens\,&{ca_{p+1}c}+s_ps_{p-1}s_p\tens{cb_{p-1}c}
=s_{p+1}s_ps_{p+1}\tens{a_{p+1}ca_{p+1}}\\
&+s_{p-1}s_ps_{p+1}\tens{b_{p-1}ca_{p+1}}
+s_{p+1}s_ps_{p-1}\tens{a_{p+1}cb_{p-1}}
+s_{p-1}s_ps_{p-1}\tens{b_{p-1}cb_{p-1}}.
\end{split}
\end{equation*}
This formula simplifies by \eqref{haczyk} and \eqref{haczyk'} to read
\begin{equation*}
\begin{split}
s_ps_{p+1}s_p\tens{ca_{p+1}c}&+s_ps_{p-1}s_p\tens{cb_{p-1}c}\\
&=s_{p+1}s_ps_{p+1}\tens{a_{p+1}ca_{p+1}}+s_{p-1}s_ps_{p-1}\tens{b_{p-1}cb_{p-1}}.
\end{split}
\end{equation*}
Applying the same reasoning to the formula $\alpha(s_{p-1}s_ps_{p-1})=\alpha(s_ps_{p-1}s_p)$ we get
\begin{equation*}
\begin{split}
s_ps_{p-1}s_p\tens{ca_{p-1}c}&+s_ps_{p+1}s_p\tens{cb_{p+1}c}\\
&=s_{p-1}s_ps_{p-1}\tens{a_{p-1}ca_{p-1}}+s_{p+1}s_ps_{p+1}\tens{b_{p+1}cb_{p+1}}.
\end{split}
\end{equation*}
Thus we obtain:
\begin{subequations}\label{cbraid}
\begin{align}
ca_{p+1}c&=a_{p+1}ca_{p+1},\label{aca'}\\
cb_{p-1}c&=b_{p-1}cb_{p-1},\nonumber\\
ca_{p-1}c&=a_{p-1}ca_{p-1},\label{aca}\\
cb_{p+1}c&=b_{p+1}cb_{p+1}.\nonumber
\end{align}
\end{subequations}

Before stating the next theorem let us adopt the following convention: in the case of even $n=2p$ we will \emph{define} $a_p$ and $b_p$ as
\begin{equation}\label{Defab}
a_p=ca_1^2,\qquad{b_p=cb_1^2}.
\end{equation}
Since $a_1^2+b_1^2=\I$ we have $a_p+b_p=c$. Note that with this definition of $a_p$ and $b_p$ formula \eqref{alab} applies universally --- regardless of parity of $n$. From now on whenever we refer to either
\begin{equation*}
\bigl\{a_1,\dotsc,a_{n-1}\bigr\}\quad\text{or}\quad
\bigl\{b_1,\dotsc,b_{n-1}\bigr\}
\end{equation*}
we mean the elements on the diagonal/anti-diagonal of $\bu$ (in case of odd $n$) or the elements on the diagonal/anti-diagonal of $\bu$ with $a_p$  or $b_p$ (as defined in \eqref{Defab}) instead of $c$ (in case of even $n$).

\begin{theorem}\label{glow}
\noindent\begin{enumerate}
\item\label{iloczyny} For any $i,j\in\{1,\dotsc,n-1\}$ we have $a_ib_j=b_ja_i=0$,
\item\label{cproj} for any $j\in\{1,\dotsc,n-1\}$ the elements $a_i^2$ and $b_i^2$ are central projections in $A$,
\item\label{cproj2} we have $a_1^2=a_2^2=\dotsm=a_{n-1}^2$ and $b_1^2=b_2^2=\dotsm=b_{n-1}^2=\I-a_1^2$,
\item\label{braid} the elements of $\{a_1,\dotsc,a_{n-1}\}$ and $\{b_1,\dotsc,b_{n-1}\}$ satisfy
\begin{enumerate}
\item $a_ia_{i+1}a_i=a_{i+1}a_ia_{i+1}$ and $b_ib_{i+1}b_i=b_{i+1}b_ib_{i+1}$ for all $i\in\{1,\dotsc,n-1\}$,
\item $a_ia_j=a_ja_i$ and $b_ib_j=b_jb_i$ whenever $|i-j|>1$.
\end{enumerate}
\end{enumerate}
\end{theorem}

\begin{proof}
Ad \eqref{iloczyny}. Let us first look at the case $n=2p+1$. We have for $i\in\{1,\dotsc,n-2\}$
\begin{equation*}
a_i=a_i(a_{i+1}^2+b_{i+1}^2)=a_ia_{i+1}^2+a_ib_{i+1}^2
\end{equation*}
and since $a_ib_{i+1}=0$ (by \eqref{konik1} and \eqref{blpion}), we get
\begin{equation}\label{aii}
a_i=a_ia_{i+1}^2.
\end{equation}
Similarly for $k\in\{2,\dotsc,n-1\}$ we have
\begin{equation}\label{akk}
a_k=a_ka_{k-1}^2.
\end{equation}
Consider now the product $a_ib_j$ for some $i,j\in\{1,\dotsc,n-1\}$. We have several cases
\begin{itemize}
\item {$i=j<n-1$}: then $a_ib_j=a_ia_{i+1}^2b_i=0$ because $a_{i+1}b_i=0$,
\item {$i=j=n-1$}: then $a_ib_j=a_ia_{i-1}^2b_i=0$ because $a_{i-1}b_i=0$,
\item {$|i-j|=1$}: then $a_ib_j=0$ by either \eqref{Konik}, \eqref{Konik'}, \eqref{blpoz} or \eqref{blpion},
\item {$i<j-1$}: then $a_ib_j=a_ia_{i+1}^2b_j=\dotsm=a_ia_{i+1}^2\dotsm{a_{j-1}^2}b_j=0$ because ${a_{j-1}^2}b_j=0$,
\item {$i>j+1$}: then $a_ib_j=a_ia_{i-1}^2b_j=\dotsm=a_ia_{i-1}^2\dotsm{a_{j+1}^2}b_j=0$ because ${a_{j+1}^2}b_j=0$.
\end{itemize}
A similar reasoning (or usage of the antipode) leads to the conclusion that $b_ja_i$ is also zero. This proves \eqref{iloczyny} in the case of odd $n$.

In the case of $n=2p$ equations \eqref{aii} and \eqref{akk} are verified in an analogous manner for
\begin{equation*}
\begin{split}
i\in\bigl\{1,\dots,p-2\bigr\},&\qquad{k}\in\bigl\{2,\dotsc,p-1\bigr\}\\
&\text{and}\\
i\in\bigl\{p+1,\dots,n-2\bigr\},&\qquad{k}\in\bigl\{p+2,\dotsc,n-1\bigr\}.
\end{split}
\end{equation*}
It follows that the only products which might potentially be non-zero are
\begin{itemize}
\item[$\blacktriangleright$] product between elements from $\{a_{p-1},a_{p+1}\}$ and $\{b_{p-1},b_{p+1}\}$,
\item[$\blacktriangleright$] products of any $a_i$ with $b_p$ and any $b_i$ with $a_p$.
\end{itemize}

Before addressing these products let us quickly see that $a_1^2$ and $b_1^2$ commute with $c$. If $p>2$ then it obviously follows from the first equation of \eqref{commut}. Now we prove this for $p=2$, or in other words, $a_1=a_{p-1}$. Since $a_{p-1}^2+b_{p-1}^2=\I$, it is enough to prove that $a_{p-1}^2c=ca_{p-1}^2$. Remembering that $c^2=\I$ we compute using \eqref{aca}
\begin{equation}\label{apm}
\begin{split}
a_{p-1}^2&=a_{p-1}c^2a_{p-1}c=a_{p-1}c\,ca_{p-1}c\\
&=a_{p-1}ca_{p-1}\,ca_{p-1}=ca_{p-1}c\,ca_{p-1}=ca_{p-1}^2.
\end{split}
\end{equation}
Note that the same technique shows also that $a_{p+1}^2$ and $b_{p+1}^2$ commute with $c$, i.e.
\begin{equation}\label{apm'}
a_{p+1}^2c=ca_{p+1}^2.
\end{equation}

We now return to products of the first type described above. We have by \eqref{aca} (or \eqref{aca'}) and \eqref{apm} (or \eqref{apm'})
\begin{equation}\label{aa3}
\begin{split}
a_{p\pm1}&=c^2a_{p\pm1}c^2=c\,ca_{p\pm1}c\,c\\
&=ca_{p\pm1}c\,a_{p\pm1}c=a_{p\pm1}ca_{p\pm1}^2c\\
&=a_{p\pm1}a_{p\pm1}^2c^2=a_{p\pm1}^3.
\end{split}
\end{equation}

Now we have
\[
a_{p\pm1}=a_{p\pm1}(a_{p\pm1}^2+b_{p\pm1}^2)=a_{p\pm1}^2+a_{p\pm1}b_{p\pm1}^2.
\]
By \eqref{aa3}, $a_{p\pm1}b_{p\pm1}^2=0$, or taking adjoints of both sides, $b_{p\pm1}^2a_{p\pm1}=0$. Therefore, by Lemma \ref{zero_repr} we have
\begin{equation}\label{bazer}
b_{p\pm1}a_{p\pm1}=0.
\end{equation}
Starting with
\[
a_{p\pm1}=(a_{p\pm1}^2+b_{p\pm1}^2)a_{p\pm1}
\]
we obtain 
\begin{equation}\label{abzer}
a_{p\pm1}b_{p\pm1}=0.
\end{equation}
Applying the antipode to \eqref{bazer} and \eqref{abzer} we find that
\begin{equation*}
a_{p\pm1}b_{p\mp1}=0\qquad\text{and}\qquad{b_{p\mp1}a_{p\pm1}=0}.
\end{equation*}

Having established that
\begin{equation*}
a_1^2c-ca_1^2=b_1^2c-cb_1^2=0
\end{equation*}
we easily see that not only
\begin{equation*}
a_pb_{r}=ca_1^2b_{r}=0,\qquad{1\leq{r}\leq{n-1}},
\end{equation*}
but also
\begin{equation*}
b_{r}a_p=b_{r}ca_1^2=b_ra_1^2c=0,\qquad{1\leq{r}\leq{n-1}}.
\end{equation*}
This finishes the proof of \eqref{iloczyny}.

Ad \eqref{braid}. Take first  $i\in\{1,\dotsc,n-2\}$ such that $i,i+1\neq\tfrac{n}{2}$ (if this conditions is applicable). The formula $\alpha(s_is_{i+1}s_i)=\alpha(s_{i+1}s_is_{i+1})$ leads to
\begin{equation*}
\begin{split}
s_{i}&s_{i+1}s_{i}\tens{a_{i}a_{i+1}a_{i}}
+s_{i}s_{i+1}s_{n-i}\tens{a_{i}a_{i+1}b_{n-i}}\\
&\quad+s_{i}s_{n-i-1}s_{i}\tens{a_{i}b_{n-i-1}a_{i}}
+s_{i}s_{n-i-1}s_{n-i}\tens{a_{i}b_{n-i-1}b_{n-i}}\\
&\quad+s_{n-i}s_{i+1}s_{i}\tens{b_{n-i}a_{i+1}a_{i}}
+s_{n-i}s_{i+1}s_{n-i}\tens{b_{n-i}a_{i+1}b_{n-i}}\\
&\quad+s_{n-i}s_{n-i-1}s_{i}\tens{b_{n-i}b_{n-i-1}a_{i}}
+s_{n-i}s_{n-i-1}s_{n-i}\tens{b_{n-i}b_{n-i-1}b_{n-i}}\\
=
\,&\,s_{i+1}s_{i}s_{i+1}\tens{a_{i+1}a_{i}a_{i+1}}
+s_{i+1}s_{i}s_{n-i-1}\tens{a_{i+1}a_{i}b_{n-i-1}}\\
&\quad+s_{i+1}s_{n-i}s_{i+1}\tens{a_{i+1}b_{n-i}a_{i+1}}
+s_{i+1}s_{n-i}s_{n-i-1}\tens{a_{i+1}b_{n-i}b_{n-i-1}}\\
&\quad+s_{n-i-1}s_{i}s_{i+1}\tens{b_{n-i-1}a_{i}a_{i+1}}
+s_{n-i-1}s_{i}s_{n-i-1}\tens{b_{n-i-1}a_{i}b_{n-i-1}}\\
&\quad+s_{n-i-1}s_{n-i}s_{i}\tens{b_{n-i-1}b_{n-i}a_{i}}
+s_{n-i-1}s_{n-i}s_{n-i-1}\tens{b_{n-i-1}b_{n-i}b_{n-i-1}}.
\end{split}
\end{equation*}
Since by \eqref{iloczyny} all products between $a_i$'s and $b_j$'s are zero, all the middle terms on both sides are equal to $0$. Thus we obtain
\begin{equation*}
\begin{split}
s_{i}s_{i+1}&s_{i}\tens{a_{i}a_{i+1}a_{i}}
+s_{n-i}s_{n-i-1}s_{n-i}\tens{b_{n-i}b_{n-i-1}b_{n-i}}\\
&=
s_{i+1}s_{i}s_{i+1}\tens{a_{i+1}a_{i}a_{i+1}}
+s_{n-i-1}s_{n-i}s_{n-i-1}\tens{b_{n-i-1}b_{n-i}b_{n-i-1}}.
\end{split}
\end{equation*}
We have explicitly ruled out the possibility $i=\frac{n}{2}$, so $s_{i}s_{i+1}s_{i}\neq{s_{n-i}s_{n-i-1}s_{n-i}}$
in all cases except when $i+1=n-i$, which means that $n=2p+1$ and $i=p$. Therefore, by linear independence of group elements inside $\CC[S_n]$, we obtain
\begin{equation*}
a_{i}a_{i+1}a_{i}=a_{i+1}a_{i}a_{i+1}\qquad\text{and}\qquad
b_{n-i}b_{n-i-1}b_{n-i}=b_{n-i-1}b_{n-i}b_{n-i-1}.
\end{equation*}
when $i\neq\tfrac{n-1}{2}$.

In the case $n=2p+1$, $i=p$ we obtain a seemingly weaker relation
\begin{equation}\label{weaker}
a_{p}a_{p+1}a_{p}+b_{p+1}b_{p}b_{p+1}=a_{p+1}a_{p}a_{p+1}+b_{p}b_{p+1}b_{p},
\end{equation}
but applying to both sides the antipode yields
\begin{equation}\label{weaker'}
a_{p}a_{p+1}a_{p}+b_pb_{p+1}b_p=a_{p+1}a_{p}a_{p+1}+b_{p+1}b_pb_{p+1}.
\end{equation}
Comparing \eqref{weaker} and \eqref{weaker'} we obtain
\begin{equation*}
a_{p}a_{p+1}a_{p}=a_{p+1}a_{p}a_{p+1}\qquad\text{and}\qquad
b_{p+1}b_{p}b_{p+1}=b_{p}b_{p+1}b_{p}.
\end{equation*}

The last case to consider is the case $n=2p$ and $i=p-1$ or $i=p$. This we easily obtain from \eqref{cbraid} by putting in $c=a_p+b_p$ and using the fact that products between $a_i$'s and $b_j$'s are zero.

The commutativity of $a_i$ and $a_j$ (as well as $b_i$ with $b_j$) for $|i-j|>1$ follows from \eqref{commut}.

Ad \eqref{cproj} and \eqref{cproj2}. Take any $i,j\in\{1,\dotsc,n-1\}$. We have
\begin{equation}\label{iaj}
a_j=a_j(a_i^2+b_i^2)=a_ja_i^2\qquad\text{and}\qquad{a_j=(a_i^2+b_i^2)a_j=a_i^2a_j.}
\end{equation}
Since $a_i$ clearly commutes with all $b_j$'s we see that $a_i^2$ is central for all $i$. Since $b_i^2=\I-a_i^2$, the elements $b_i^2$ are also central. Putting $j=i$ in \eqref{iaj} we get $a_i=a_i^3$, so $a_i^2$ is a projection (and thus so is $b_i^2$). This proves \eqref{cproj}.

Using \eqref{iaj} once for $(i,j)$ and once for $(j,i)$ we find that
\[
a_j^2=a_j^2a_i^2=a_i^2
\]
which finishes the proof of \eqref{cproj2}.
\end{proof}

Equipped with the conclusion of Theorem \ref{glow} we easily see that the subalgebra of $\cA$ generated by the entries of $\bu$ is a quotient of the direct sum $\CC[S_n]\oplus\CC[S_n]$. Let $\balpha\colon\CC[S_n]\to\CC[S_n]\tens\bigl(\CC[S_n]\tens\CC[S_n]\bigr)$ be given by
\begin{equation}\label{Defbal}
\balpha(s_i)=s_i\tens\sigma_i+s_{n-i}\tens\tau_{n-i},
\end{equation}
where $\sigma_1,\dots,\sigma_{n-1}$ are the generators $s_1,\dotsc,s_{n-1}$ in the first copy of $\CC[S_n]$ inside $\CC[S_n]\oplus\CC[S_n]$ and by $\tau_1,\dots,\tau_{n-1}$ are the same generators in the second copy. Then $\balpha$ is a coaction of
\begin{equation}\label{KK}
\KK=\bigl(\CC[S_n]\oplus\CC[S_n],\Delta\bigr)
\end{equation}
with
\begin{equation*}
\begin{split}
\Delta(\sigma_i)&=\sigma_i\tens\sigma_i+\tau_i\tens\tau_{n-i},\\
\Delta(\tau_i)&=\sigma_i\tens\tau_i+\tau_i\tens\sigma_{n-i}
\end{split}
\end{equation*}
on $\CC[S_n]$. We will show in Subsection \ref{isom} that this coaction is isometric in the sense described in Subsection \ref{SpTr}. Moreover, by Theorem \ref{glow}, for any isometric $\alpha$ of $\GG=(A,\Delta_A)$ there is a unique $\Phi\colon\CC[S_n]\oplus\CC[S_n]\to{A}$ such that
\begin{equation*}
\Phi(a_i)=\sigma_i\quad\text{and}\quad\Phi(b_i)=\tau_i,\qquad{1}\leq{i}\leq{n-1}.
\end{equation*}
Finally $\Phi$ is easily seen to be a quantum group morphism and to satisfy
\begin{equation}\label{Phial}
\alpha=(\id\tens\Phi)\comp\balpha.
\end{equation}

It follows that $\KK$ is the universal compact quantum group coacting isometrically on the spectral triple $\bigl(\CC[S_n],\ell^2(S_n),D\bigr)$.

In the next section we will describe a different way to obtain the quantum group 
\begin{equation*}
\QISO\bigl(\CC[S_n],\ell^2(S_n),D\bigr)=\KK=\bigl(\CC[S_n]\oplus\CC[S_n],\Delta\bigr)
\end{equation*}
from $\CC[S_n]$.

\subsection{The coaction \texorpdfstring{$\balpha$}{a}}\label{isom}

In this subsection we quickly verify that the coaction $\balpha$ of \eqref{KK} defined in \eqref{Defbal} is isometric. Let $\theta$ be the automorphism of $\CC[S_n]$ mapping $s_i$ to $s_{n-1}$ and let $\beta$ be the automorphism of $\CC[S_n]\oplus\CC[S_n]$ given by
\begin{equation*}
\beta(\sigma_i)=\sigma_{n-i},\qquad\beta(\tau_j)=\tau_{n-j}
\end{equation*}
(so $\beta$ is the automorphism $\theta$ acting in both copies of $\CC[S_n]$ simultaneously). Moreover let $\psi$ be the algebra homomorphism
\begin{equation*}
\psi\colon\CC[S_n]\to\CC[S_n]\tens\bigl(\CC[S_n]\oplus\CC[S_n]\bigr)
\end{equation*}
given on generators by
\begin{equation*}
\psi(s_i)=s_i\tens\sigma_i.
\end{equation*}
Then
\begin{equation}\label{blpsi}
\balpha(x)=\psi(x)+(\theta\tens\beta)\bigl(\psi(x)\bigr)
\end{equation}
for all $x\in\CC[S_n]$. If $\Hat{D}$ is the Dirac operator of the spectral triple $\bigl(\CC[S_n],\ell^2(S_n),D\bigr)$ considered as a map $\CC[S_n]\to\CC[S_n]$ (cf.~Subsection \ref{SpTr}) then we easily see that
\begin{equation*}
\psi\comp\Hat{D}=(\Hat{D}\tens\I)\comp\psi.
\end{equation*}
Also, we clearly have $\Hat{D}\comp\theta=\theta\comp\Hat{D}$, so by \eqref{blpsi} we get
\begin{equation*}
\balpha\comp\Hat{D}=(\Hat{D}\tens\I)\comp\balpha.
\end{equation*}

Recall that in order to derive conditions on the matrix elements of $\bu$ in Section \ref{qiso} we only assumed that the considered coaction preserves the subspace of elements of length $1$ (Eq.~\eqref{alsj}) and the canonical trace $\btau$ is preserved under $\alpha$ on elements of length $0$ and $2$ (eq.~\eqref{deltakl}). Nevertheless, the conclusion was that $\alpha$ is related to $\balpha$ via formula \eqref{Phial}. Since $\balpha$ is isometric, it follows that the weaker conditions on $\alpha$ assumed in the beginning of Section \ref{qiso} actually imply the stronger condition of being an isometric coaction.

\section{Doubling procedure}\label{doubling}

In this section we will describe a simple procedure which out of a regular multiplier Hopf algebra with invariant functionals endowed with an order two automorphism produces a new regular multiplier Hopf algebra with invariant functionals. This will be called the \emph{doubling procedure}. The name will be justified by the example presented in Subsection \ref{exCCG}.

We refer to the standard sources \cite{mha,afgd} for the theory of multiplier Hopf algebras and their duality. For an algebra $\cA$ (always over $\CC$) with non-degenerate product the symbol $\cM(\cA)$ will denote the \emph{algebraic multiplier algebra} of $\cA$ (cf.~\cite{mha}).

\subsection{The procedure}\label{doublB}

Let $(\cA,\Delta)$ be a regular multiplier Hopf algebra with invariant functionals (cf.~\cite{afgd} for the definition of such objects) and let $\theta\colon\cA\to\cA$ be an automorphism of $(\cA,\Delta)$, i.e.~an automorphism of $\cA$ such that
\begin{equation*}
(\theta\tens\theta)\comp\Delta=\Delta\comp\theta
\end{equation*}
(while it is quite easy to show that $\theta(\cA)=\cA$, on the left hand side one needs to consider the unique extension of $(\theta\tens\theta)$ to a homomorphism $\cM(\cA\tens\cA)\to\cM(\cA\tens\cA)$). The first important fact is that the adjoint of the map $\theta$ on the linear dual of $\cA$ restricts to a map $\Hat{\cA}\to\Hat{\cA}$. Indeed, this follows from the fact that if $\ph$ is a non-zero left invariant functional on $\cA$ then so is, $\ph\comp\theta$. By the uniqueness of invariant functionals $\ph\comp\theta=c\ph$ for some non zero constant $c$. It follows that for any $a\in\cA$
\begin{equation*}
\ph(a\,\cdot\:\!)\comp\theta=\ph\bigl(c\theta^{-1}(a)\,\cdot\:\!\bigr).
\end{equation*}
This way we define the \emph{dual} of $\theta$ which is in fact an automorphism of $(\Hat{\cA},\Hat{\Delta})$. We will denote this automorphism by $\Hat{\theta}$.

Now let us consider the special case when $\theta$ is an order-two automorphism. Let $\cB$ be the crossed product $\Hat{\cA}\rtimes_{\Hat{\theta}}\ZZ_2$. The algebra $\cB$ can be identified with the algebra of matrices
\begin{equation*}
\left\{\begin{bmatrix}x&y\\\Hat{\theta}(y)&\Hat{\theta}(x)\end{bmatrix}\st{x,y}\in\Hat{\cA}\,\right\}.
\end{equation*}
(see e.g.~\cite[Part $\mathrm{I}$, Example 2.11]{cont} and \cite[Remark before Proposition 7.5]{kvd}). Note that the mapping
\begin{equation*}
\imath\colon\Hat{\cA}\ni{x}\longmapsto\begin{bmatrix}x&0\\0&\Hat{\theta}(x)\end{bmatrix}\in\cB
\end{equation*}
is a non-degenerate homomorphism.\footnote{For this one should use the \emph{local units} of $\cA$, cf.~\cite[Proposition 2.6]{dvdz}.}  Moreover there is a multiplier $U$ of $\cB$ such that $U^2=\I$ and
\begin{equation*}
U\imath(x)U=\imath\bigl(\Hat{\theta}(x)\bigr)
\end{equation*}
for all $x\in\Hat{\cA}$. Indeed, $U=\bigl(\begin{smallmatrix}0&\I\\\I&0\end{smallmatrix}\bigr)$. If $\cC$ is another algebra with non-degenerate product equipped with a non-degenerate mapping $\jmath\colon\Hat{\cA}\to\cM(\cC)$ and an element $V\in\cM(\cC)$ such that $V^2=\I$ and
\begin{equation*}
V\jmath(x)V=\jmath\bigl(\Hat{\theta}(x)\bigr)
\end{equation*}
for all $x\in\Hat{\cA}$ then there exists a unique non-degenerate map $\Lambda\colon\cB\to\cM(\cC)$ such that
\begin{equation*}
\Lambda\bigl(\imath(x)\bigr)=\jmath(x)
\end{equation*}
for all $x\in\Hat{\cA}$ and $\Lambda(U)=V$. Indeed, uniqueness of $\Lambda$ is immediate, since any element of $\cB$ is a sum of elements of the form $\imath(x)$ with $x\in\Hat{\cA}$ and $\imath(y)U$ with $y\in\Hat{\cA}$. On the other hand if $\jmath$ and $V$ are as above then the algebra generated inside $\cM(\cC)$ by the image of $\jmath$ and by $V$ is clearly an image of the matrix algebra $\cB$. Non-degeneracy of $\Lambda$ follows from the non-degeneracy of $\jmath$ (the image of $\Lambda$ contains the image of $\jmath$).

Now consider $\cC=\cB\tens\cB$ with $\jmath\colon\Hat{\cA}\ni{x}\mapsto(\imath\tens\imath)\Hat{\Delta}(x)$ and $V=U\tens{U}$. Let the corresponding unique map $\cB\to\cM(\cB\tens\cB)$ be denoted by $\Delta_{\cB}$. It is straightforward that $\Delta_{\cB}$ is a comultiplication which makes $\cB$ into a regular multiplier Hopf algebra. The counit and antipode of $\cB$ are given by
\begin{equation*}
\epsilon_{\cB}\left(
\begin{bmatrix}x&y\\\Hat{\theta}(y)&\Hat{\theta}(x)\end{bmatrix}\right)=\Hat{\epsilon}(x),\qquad
S_{\cB}\left(
\begin{bmatrix}x&y\\\Hat{\theta}(y)&\Hat{\theta}(x)\end{bmatrix}\right)=
\begin{bmatrix}\Hat{S}(x)&\Hat{S}(y)\\
\Hat{\theta}\bigl(\Hat{S}(y)\bigr)&\Hat{\theta}\bigl(\Hat{S}(x)\bigr)\end{bmatrix},
\end{equation*}
where $\Hat{\epsilon}$ and $\Hat{S}$ are counit and antipode of $(\Hat{\cA},\Hat{\Delta})$.\footnote{To see that $\epsilon_\cB$ and $S_\cB$ satisfy the desired properties one must use the simple fact that $\Hat{\epsilon}\comp\Hat{\theta}=\Hat{\epsilon}$ and $\Hat{S}\comp\Hat{\theta}=\Hat{\theta}\comp\Hat{S}=\Hat{S}$.
}

The regular multiplier Hopf algebra $(\cB,\Delta_{\cB})$ is easily seen to admit invariant functionals: if $\Hat{\ph}$ and $\Hat{\psi}$ are left and right invariant functionals on $\Hat{\cA}$ then
\begin{equation*}
\ph_\cB\left(
\begin{bmatrix}x&y\\\Hat{\theta}(y)&\Hat{\theta}(x)\end{bmatrix}\right)=\Hat{\ph}(x)\quad\text{and}\quad
\psi_\cB\left(\begin{bmatrix}x&y\\\Hat{\theta}(y)&\Hat{\theta}(x)\end{bmatrix}\right)=\Hat{\psi}(x)
\end{equation*}
define invariant functionals on $(\cB,\Delta_\cB)$.

The \emph{doubling} $(\widetilde{\cA},\widetilde{\Delta})$ of $(\cA,\Delta,\theta)$ is by definition the multiplier Hopf algebra dual to $(\cB,\Delta_\cB)$. One can easily see that $(\cA,\Delta)$ is a \emph{quotient} of $(\widetilde{\cA},\widetilde{\Delta})$. In fact, the construction of the doubling is fully analogous to the operation of semidirect product by $\ZZ_2$ in group theory. In particular it does not lead out of the class of commutative multiplier Hopf algebras.

\subsection{Example}\label{exCCG}

Now let us apply the construction discussed above to the case of a group algebra $\CC[G]$ of a discrete group $G$ with an order two automorphism $\theta$. The dual $\Hat{\cA}$ of $\cA$ is then the algebra of finitely supported functions on $G$ with comultiplication $\Hat{\Delta}$ coming from group multiplication:
\begin{equation*}
\Hat{\Delta}(f)(x,y)=f(xy)
\end{equation*}
for $f\in\Hat{\cA}$ and $x,y\in{G}$. The delta functions $\{\delta_x\}_{x\in{G}}$ for a basis of $\Hat{\cA}$ and
\begin{equation*}
\Hat{\Delta}(\delta_x)=\sum_{ab=x}\delta_a\tens\delta_b
\end{equation*}
where the possibly infinite sum clearly defines a multiplier of $\Hat{\cA}\tens\Hat{\cA}$. Note that the multiplication in $\Hat{\cA}$ is the usual pointwise multiplication of functions on $G$, so $\delta_x$ multiplied by $\delta_y$ is non-zero if and only if $x=y$ and then the product is $\delta_x$.

The functional
\begin{equation*}
\Hat{\cA}\ni{f}\longmapsto\sum_{x\in{G}}f(x)
\end{equation*}
is both left and right invariant. The crossed product algebra $\cB$ defined in Subsection \ref{doublB} is isomorphic as an algebra with $\Hat{\cA}\oplus\Hat{\cA}$ with multiplication
\begin{equation*}
(f,g)(u,v)=\bigl(fu+g\Hat{\theta}(v),fv+g\Hat{\theta}(u)\bigr),
\end{equation*}
where $\Hat{\theta}(u)=u\comp\theta$ for any $u\in\Hat{\cA}$. The inclusion $\imath\colon\Hat{\cA}\hookrightarrow\cM(\cB)$ is given by $f\mapsto(f,0)$, while the multiplier $U$ is $(0,\I)$. Suppressing $\imath$ we can treat $\delta_x$ and $\delta_yU$ as an elements $(\delta_x,0)$ and $(0,\delta_y)$ of $\cB$. With this notation we have for any $x,y\in{G}$
\begin{equation*}
\Delta_\cB(\delta_x)=\sum_{ab=x}\delta_a\tens\delta_b,\qquad
\Delta_\cB(\delta_yU)=\sum_{ab=y}\delta_aU\tens\delta_bU.
\end{equation*}
The set $\{\delta_x\}_{x\in{G}}\cup\{\delta_yU\}_{y\in{G}}$ is a basis of $\cB$.

Let $h$ be the left and right invariant functional on $\cB$
\begin{equation*}
h\bigl((f,g)\bigr)=\sum_{x\in{G}}f(x).
\end{equation*}
Now consider the functionals $\{\xi_x\}_{x\in{G}}$ and $\{\eta_y\}_{y\in{G}}$ defined by
\begin{equation*}
\xi_x=h(\,\cdot\,\delta_x),\qquad\text{and}\qquad\eta_y=h(\,\cdot\,U\delta_y).
\end{equation*}
By construction they belong to $\widetilde{\cA}=\Hat{\cB}$ and one immediately sees that they span this space. Moreover it is easy to see that
\begin{equation*}
\xi_x(\delta_z)=
\begin{cases}
1&z=x,\\0&\text{otherwise},
\end{cases}
\qquad
\eta_y(\delta_zU)=
\begin{cases}
1&z=y,\\0&\text{otherwise}
\end{cases}
\end{equation*}
and
\begin{equation*}
\xi_x(U\delta_z)=0=\eta_y(\delta_z)
\end{equation*}
for all $x,y,z\in{G}$.

Multiplication in $\widetilde{\cA}$ is dual to comultiplication in $\cB$. It follows from a simple computation that
\begin{equation*}
\xi_x\xi_y=\xi_{xy},\quad\xi_x\eta_y=\eta_y\xi_x=0,\quad
\eta_x\eta_y=\eta_{xy}.
\end{equation*}
Therefore $\widetilde{\cA}$ is, as an algebra, isomorphic to the direct sum of two copies of the group algebra of $G$:
\begin{equation*}
\widetilde{\cA}\cong\CC[G]\oplus\CC[G]
\end{equation*}
with $\{\xi_x\}_{x\in{G}}$ corresponding to the canonical basis of the first copy of $\CC[G]$ and $\{\eta_y\}_{y\in{G}}$ corresponding to the canonical basis of the second copy. Let us note that, in particular, $\widetilde{\cA}$ is unital.

We shall now describe the comultiplication on $\widetilde{\cA}$. This is very simple, since the remark after \cite[Definition 4.4]{afgd} says that for any $\xi\in\widetilde{\cA}=\Hat{\cB}$ we have
\begin{equation*}
\widetilde{\Delta}(\xi)\,(\Xi\tens\Theta)=\xi(\Xi\Theta).
\end{equation*}
for all $\Xi,\Theta\in\cB$. Since
\begin{equation*}
\begin{split}
&\widetilde{\Delta}(\xi_x)(\delta_y\tens\delta_z)=h(\delta_y\delta_z\delta_x)=
\begin{cases}
1&x=y=z,\\
0&\text{otherwise},
\end{cases}\\
&\widetilde{\Delta}(\xi_x)(\delta_y\tens\delta_zU)=h(\delta_y\delta_zU\delta_x)=0,\\
&\widetilde{\Delta}(\xi_x)(\delta_yU\tens\delta_z)=h(U\delta_y\delta_z\delta_x)=0,\\
&\widetilde{\Delta}(\xi_x)(\delta_yU\tens\delta_zU)=h(\delta_yU\delta_zU\delta_x)=
h(\delta_y\delta_{\theta(z)}\delta_x)=
\begin{cases}
1&x=y=\theta(z),\\
0&\text{otherwise}
\end{cases}
\end{split}
\end{equation*}
and
\begin{equation*}
\begin{split}
&\widetilde{\Delta}(\eta_x)(\delta_y\tens\delta_z)=h(\delta_y\delta_zU\delta_x)=0,\\
&\widetilde{\Delta}(\eta_x)(\delta_y\tens\delta_zU)=h(\delta_y\delta_z\delta_x)=
\begin{cases}
1&x=y=z,\\
0&\text{otherwise},
\end{cases}\\
&\widetilde{\Delta}(\eta_x)(\delta_yU\tens\delta_z)=h(\delta_yU\delta_zU\delta_x)=
h(\delta_y\delta_{\theta(z)}\delta_x)=
\begin{cases}
1&x=y=\theta(z),\\
0&\text{otherwise}
\end{cases}\\
&\widetilde{\Delta}(\eta_x)(\delta_yU\tens\delta_zU)=h(\delta_yU\delta_z\delta_x)=0
\end{split}
\end{equation*}
we have
\begin{equation*}
\begin{split}
\widetilde{\Delta}(\xi_x)&=\xi_x\tens\xi_x+\eta_x\tens\eta_{\theta(x)},\\
\widetilde{\Delta}(\eta_x)&=\xi_x\tens\eta_x+\eta_x\tens\xi_{\theta(x)}.
\end{split}
\end{equation*}
If we still denote by $\theta$ the canonical extension of $\theta$ to an automorphism of $\widetilde{\cA}=\CC[G]\oplus\CC[G]$ then we can write
\begin{equation}\label{DelTil}
\begin{split}
\widetilde{\Delta}(\xi_x)&=\xi_x\tens\xi_x+\eta_x\tens\theta(\eta_x),\\
\widetilde{\Delta}(\eta_x)&=\xi_x\tens\eta_x+\eta_x\tens\theta(\xi_x).
\end{split}
\end{equation}

Let us specify the situation further and take $G=S_n$. As in previous sections let us denote by $\sigma_1,\dots,\sigma_{n-1}$ the transpositions $s_1,\dotsc,s_{n-1}$ in the first copy of $\CC[S_n]$ inside $\CC[S_n]\oplus\CC[S_n]$ and by $\tau_1,\dots,\tau_{n-1}$ the same transpositions in the second copy. Let $\theta$ be the automorphism of $S_n$ given by conjugation with the unique word of maximal length (on the considered generators). Then $\theta(\sigma_i)=\sigma_{n-i}$ and $\theta(\tau_i)=\tau_{n-i}$ for all $i$. In particular the formula \eqref{DelTil} gives the comultiplication $\Delta$ on $\CC[S_n]\oplus\CC[S_n]$ obtained via the doubling procedure for $\theta$ as
\begin{equation*}
\begin{split}
\Delta(\sigma_i)&=\sigma_i\tens\sigma_i+\tau_i\tens\tau_{n-i},\\
\Delta(\tau_i)&=\sigma_i\tens\tau_i+\tau_i\tens\sigma_{n-i}.
\end{split}
\end{equation*}
This means that the matrix
\begin{equation*}
\begin{bmatrix}
\sigma_1  &          &       &       &            &\tau_1      \\
          &\sigma_2  &       &       &\tau_2      &            \\
          &          &\ddots &\iddots&            &            \\
          &          &\iddots&\ddots &            &            \\
          &\tau_{n-2}&       &       &\sigma_{n-2}&            \\
\tau_{n-1}&          &       &       &            &\sigma_{n-1}
\end{bmatrix}
\end{equation*}
(in case $n$ is even, the middle element of the matrix is $\sigma_{\frac{n}{2}}+\tau_{\frac{n}{2}}$)
is a corepresentation of the finite quantum group $\KK=\bigl(\CC[S_n]\oplus\CC[S_n],\Delta\bigr)$. This completes the proof of Theorem \ref{main}.

\section{The quantum isometry groups for \texorpdfstring{$S_3$}{S3} with different sets of generators}
\label{rank}

In \cite[Section 4]{bs} quantum isometry groups of spectral triples on $S_3$ were discussed. The two spectral triples were related to two different sets of generators used to define the length function on $S_3$ and the corresponding Dirac operator.

In order to formulate the precise result of the work of Bhowmick and Skalski let us denote by $A$ the \cst-algebra $\CC[S_3]\oplus\CC[S_3]$ and let $\sigma_1,\sigma_2$ and $\tau_1,\tau_2$ denote the transpositions $s_1$ and $s_2$ in the first and second copy of $\CC[S_3]$ inside $A$ respectively (just as in Sections \ref{intro}--\ref{doubling}).

Then we have
\begin{enumerate}
\item The quantum isometry group of $S_3$ with the generating set $\{s_1,s_2\}$ is
\begin{equation*}
\KK_1=(A,\Delta_1),
\end{equation*}
where\footnote{Please note that the formulas for comultiplication in this and the next case hinted upon in \cite{bs} are slightly incorrect.}
\begin{equation*}
\begin{aligned}
\Delta_1(\sigma_1)&=\sigma_1\tens\sigma_1+\tau_1\tens\tau_2,&
\Delta_1(\tau_1)&=\tau_1\tens\sigma_2+\sigma_1\tens\tau_1,\\
\Delta_1(\sigma_2)&=\sigma_2\tens\sigma_2+\tau_2\tens\tau_1,&
\Delta_1(\tau_2)&=\tau_2\tens\sigma_1+\sigma_2\tens\tau_2.
\end{aligned}
\end{equation*}
The coaction of $\KK_1$ on $\CC[S_3]$ is given on generators by
\begin{equation*}
\begin{split}
\balpha_1(s_1)&=s_1\tens\sigma_1+s_2\tens\tau_2,\\
\balpha_1(s_2)&=s_1\tens\sigma_2+s_1\tens\tau_1.
\end{split}
\end{equation*}
(This is in full analogy to the cases $n>3$ discussed in previous sections.)
\item The quantum isometry group of $S_3$ with the generating set $\{s_1,s_1s_2,s_2s_1\}$ is
\begin{equation*}
\KK_2=(A,\Delta_2),
\end{equation*}
where
\begin{equation*}
\begin{aligned}
\Delta_2(\sigma_1)&=\sigma_1\tens\sigma_1+\tau_1\tens\tau_1,&
\Delta_2(\tau_1)&=\tau_1\tens\sigma_1+\sigma_1\tens\tau_1,\\
\Delta_2(\sigma_2)&=\sigma_2\tens\sigma_2+\tau_1\tau_2\tau_1\tens\tau_2,&
\Delta_2(\tau_2)&=\tau_2\tens\sigma_2+\sigma_1\sigma_2\sigma_1\tens\tau_2.
\end{aligned}
\end{equation*}
The coaction of $\KK_2$ on $\CC[S_3]$ is given on generators by
\begin{equation*}
\begin{split}
\balpha_2(s_1)&=s_1\tens(\sigma_1+\tau_1),\\
\balpha_2(s_2)&=s_2\tens\sigma_2+s_1s_2s_1\tens\tau_2
\end{split}
\end{equation*}
(the elements $s_2$ and $s_1s_2s_1$ are, in this case, of length two).
\end{enumerate}

In both cases we obtain a non-commutative and non-cocommutative Hopf algebra of dimension 12. By the results of \cite{F} there are only two such Hopf algebras up to isomorphism. Our aim in this section is to show that $\KK_1$ is not isomorphic to $\KK_2$. This, in particular, shows that the Hopf algebras constructed by Fukuda in \cite{F} have a (non-commutative) geometric origin. We will prove that $\KK_1$ is not isomorphic to $\KK_2$ by assuming that indeed there is an automorphism $\ph$ of $A$ such that
\begin{equation}\label{Dph}
(\ph\tens\ph)\comp\Delta_1=\Delta_2\comp\ph
\end{equation}
and then showing that this leads to a contradiction.

Let $\mu$ be the multiplication map $A\tens{A}\to{A}$ and define two linear maps $T_1,T_2\colon{A}\to{A}$ as
\begin{equation}\label{Ti}
T_i=\mu\comp\Delta_i,\qquad{i=1,2}.
\end{equation}
Denote by $e_\sigma$ and $e_\tau$ the neutral elements of $S_3$ sitting in the first and second copy of $\CC[S_3]$ inside $A$. The set
\begin{equation}\label{bas}
\bigl\{e_\sigma,\sigma_1,\sigma_2,\sigma_1\sigma_2,\sigma_2\sigma_1,\sigma_1\sigma_2\sigma_1,
e_\tau,\tau_1,\tau_2,\tau_1\tau_2,\tau_2\tau_1,\tau_1\tau_2\tau_1\bigr\}
\end{equation}
is a basis of $A$. We have
\begin{equation*}
\begin{aligned}
T_1(e_\sigma)&=e_\sigma+e_\tau,&\quad&&T_2(e_\sigma)&=e_\sigma+e_\tau,\\
T_1(\sigma_1)&=e_\sigma+\tau_1\tau_2,&\quad&&T_2(\sigma_1)&=e_\sigma+\tau_1\tau_2,\\
T_1(\sigma_2)&=e_\sigma+\tau_2\tau_1,&\quad&&T_2(\sigma_2)&=e_\sigma+\tau_2\tau_1,\\
T_1(\sigma_1\sigma_2)&=\sigma_2\sigma_1+e_\tau,&\quad&&T_2(\sigma_1\sigma_2)&=\sigma_2\sigma_1+e_\tau,\\
T_1(\sigma_2\sigma_1)&=\sigma_1\sigma_2+e_\tau,&\quad&&T_2(\sigma_2\sigma_1)&=\sigma_1\sigma_2+e_\tau,\\
T_1(\sigma_1\sigma_2\sigma_1)&=e_\sigma+e_\tau,&\quad&&T_2(\sigma_1\sigma_2\sigma_1)&=e_\sigma+\tau_1\tau_2
\end{aligned}
\end{equation*}
and the values of $T_1$ and $T_2$ on the remaining vectors of the basis \eqref{bas} are zero.

As the map $\ph$ is multiplicative (i.e.~$\ph\comp\mu=\mu\comp(\ph\tens\ph)$), from \eqref{Dph} and \eqref{Ti} we infer that
\begin{equation}\label{Tph}
\ph\comp{T_1}=T_2\comp\ph.
\end{equation}

Since $e_\sigma+e_\tau=\I$ (the unit of $A$) we have $\ph\bigl(T_1(\sigma_1\sigma_2\sigma_1)\bigr)=\I$ and by \eqref{Tph} $T_2\bigl(\ph(\sigma_1\sigma_2\sigma_1)\bigr)=\I$. This means that $\ph(\sigma_1\sigma_2\sigma_1)$ differs from $e_\sigma$ by an element of $\ker{T_2}$. It is easy to check that
\begin{equation*}
\ker{T_2}=\lin\bigl\{
\sigma_2-\sigma_1\sigma_2\sigma_1,e_\tau,\tau_1,\tau_2,\tau_1\tau_2,\tau_2,\tau_1,\tau_1\tau_2\tau_1\bigr\}.
\end{equation*}
Therefore
\begin{equation}\label{no1}
\ph(\sigma_1\sigma_2\sigma_1)=e_\sigma+\lambda(\sigma_2-\sigma_1\sigma_2\sigma_1)+
\comb(e_\tau,\tau_1,\tau_2,\tau_1\tau_2,\tau_2,\tau_1,\tau_1\tau_2\tau_1),
\end{equation}
where $\comb(\dotsm)$ represents some linear combination of elements in parentheses. Furthermore
\begin{equation*}
(\sigma_1\sigma_2\sigma_1)^2=e_\sigma,
\end{equation*}
so that
\begin{equation}\label{no3}
\ph(\sigma_1\sigma_2\sigma_1)^2=\ph\bigl((\sigma_1\sigma_2\sigma_1)^2\bigr)=\ph(e_\sigma).
\end{equation}
But $T_1(e_\sigma)=\I$ as well, so $T_2\bigl(\ph(e_\sigma)\bigr)=\ph\bigl(T_1(e_\sigma)\bigr)=\I$. This means that also
\begin{equation}\label{no2}
\ph(e_\sigma)=e_\sigma+\lambda'(\sigma_2-\sigma_1\sigma_2\sigma_1)+
\comb'(e_\tau,\tau_1,\tau_2,\tau_1\tau_2,\tau_2,\tau_1,\tau_1\tau_2\tau_1).
\end{equation}
Inserting \eqref{no1} and \eqref{no2} into \eqref{no3} we get
\begin{equation*}
\begin{split}
e_\sigma+2\lambda^2e_\sigma-\lambda^2(\sigma_1\sigma_2+\sigma_2\sigma_1)&+\lambda(\sigma_1-\sigma_1\sigma_2\sigma_1)+\bigl(\comb(e_\tau,\tau_1,\tau_2,\tau_1\tau_2,\tau_2,\tau_1,\tau_1\tau_2\tau_1)\bigr)^2\\
&=
e_\sigma+\lambda'(\sigma_2-\sigma_1\sigma_2\sigma_1)+
\comb'(e_\tau,\tau_1,\tau_2,\tau_1\tau_2,\tau_2,\tau_1,\tau_1\tau_2\tau_1).
\end{split}
\end{equation*}
Using linear independence of the group elements in the group algebra we find that $\lambda=0$ and in consequence $\lambda'=0$. This reduces \eqref{no1} and \eqref{no2} to
\begin{equation*}
\begin{split}
\ph(\sigma_1\sigma_2\sigma_1)&=e_\sigma
+\comb(e_\tau,\tau_1,\tau_2,\tau_1\tau_2,\tau_2,\tau_1,\tau_1\tau_2\tau_1),\\
\ph(e_\sigma)&=e_\sigma
+\comb'(e_\tau,\tau_1,\tau_2,\tau_1\tau_2,\tau_2,\tau_1,\tau_1\tau_2\tau_1).
\end{split}
\end{equation*}
Clearly $\ph(e_\sigma)$ is a projection and so must be $\comb'(e_\tau,\tau_1,\tau_2,\tau_1\tau_2,\tau_2,\tau_1,\tau_1\tau_2\tau_1)$. However, the image of $e_\sigma$ under an automorphism cannot be a projection strictly bigger than $e_\sigma$. In order to see this we can realize $A$ concretely as a multi-matrix algebra 
\begin{equation*}
A=\CC\oplus{M_2(\CC)}\oplus\CC\oplus\CC\oplus{M_2(\CC)}\oplus\CC\subset{M_8(\CC)}
\end{equation*}
and use the fact that in an automorphism of a multi-matrix algebra leaves invariant the rank of a matrix. It follows that $\comb'(e_\tau,\tau_1,\tau_2,\tau_1\tau_2,\tau_2,\tau_1,\tau_1\tau_2\tau_1)=0$ and
\begin{equation*}
\ph(e_\sigma)=e_\sigma.
\end{equation*}
The same argument, that the rank of the matrices $\sigma_1\sigma_2\sigma_1$ and $\ph(\sigma_1\sigma_2\sigma_1)$ must be equal to $4$ (both $\sigma_1\sigma_2\sigma_1$ and $e_\sigma$ are invertible in the first summand $\CC[S_3]$ in $A$), shows that the rank of $\comb(e_\tau,\tau_1,\tau_2,\tau_1\tau_2,\tau_2,\tau_1,\tau_1\tau_2\tau_1)$ is zero, so that
\begin{equation*}
\ph(\sigma_1\sigma_2\sigma_1)=e_\sigma.
\end{equation*}
This means that $\sigma_1\sigma_2\sigma_1-e_\sigma$ belongs to the kernel of $\ph$, so that $\ph$ cannot be an automorphism of $A$.

In the paper \cite{F} the list of all $12$-dimensional semisimple Hopf algebras over algebraically closed fields of characteristic different from $2$ and $3$ was given. Apart from commutative and cocommutative ones there are only two non-commutative and non-commutative Hopf algebras in this class. These two are denoted by $A_+$ and $A_-$ respectively. As algebras they are both isomorphic to our algebra $A$. The presentation of $A_\pm$ used in \cite{F} is the following: $A_\pm$ is the algebra generated by three elements $a,b$ and $c$ such that
\begin{equation*}
c^2=c,\quad{a^3=b^2=\I},\quad{bab=a^2},\quad{ac=ca},\quad{bc=cb}.
\end{equation*}
The comultiplications of $A_\pm$ are
\begin{equation*}
\begin{aligned}
\Delta_+(a)&=ac\tens{a}+a(\I-c)\tens{a^2},&\Delta_-(a)&=\Delta_+(a),\\
\Delta_+(b)&=b\tens{b},&\Delta_-(b)&=bc\tens{b}+b(\I-c)\tens{c}(2c-\I),\\
\Delta_+(c)&=c\tens{c}+(\I-c)\tens(\I-c),&\Delta_-(c)&=\Delta_+(c).
\end{aligned}
\end{equation*}
Consider the identification of $A_+$ with $A$ defined by
\[
a\longleftrightarrow\sigma_1\sigma_2+\tau_1\tau_2,\qquad
b\longleftrightarrow{\sigma_1+\tau_1},\qquad
c\longleftrightarrow{e_\sigma}.
\]
This isomorphism of algebras is easily seen to be an anti-isomorphism of coalgebra between $A_+$ and $(A,\Delta_2)$. This means that $\KK_2$ is isomorphic to the Hopf algebra co-opposite to $A_+$. Since the isomorphism classes of $A_+$ and $A_-$ are determined by their groups of group-like elements of the respective Hopf algebras (\cite{F}), we see that $A_\pm$ is isomorphic to its co-opposite Hopf algebra. In particular $\KK_2$ is isomorphic to $A_+$ and consequently $\KK_1$ is isomorphic to $A_-$.

\subsection*{Acknowledgment}

The authors would like to thank Adam Skalski for helpful discussions and suggestions.

\end{document}